\documentclass[journal,draftclsnofoot,onecolumn,12pt]{IEEEtran}

\IEEEoverridecommandlockouts
\usepackage[update,prepend]{epstopdf}

\usepackage{multirow}
\usepackage{tikz}
\usepackage{bbm} 
\usepackage{pdfpages}
\usepackage{multirow}
\usepackage{subfig}
\usepackage{comment}
\usepackage{amssymb}

\usepackage{algorithm,algorithmic}
\usepackage{multicol}

\usepackage{amsmath}
\newtheorem{lemma}{Lemma}
\newtheorem{theorem}{Theorem}
\newtheorem{proof}{Proof}
\usepackage[justification=centering]{caption}
\usepackage{textcomp}
\usepackage{psfrag}
\usepackage{arydshln}
\usepackage{url}
\usepackage{soul}
\usepackage{multirow}
\usepackage{multicol}
\usepackage[nolist]{acronym}
\usepackage{mathtools,lipsum}
\usepackage{cuted}
\usepackage{graphics}
\usepackage{graphicx,color}
\usepackage{epstopdf}
\usepackage{epsfig}

\title{Latency Analysis of LEO Satellite Relay Communication: An Application of Conditional Contact Angle Distribution\\
\thanks{
Corresponding author: Xiang Ling, E-mail: xiangling@uestc.edu.cn. The condensed version of this article has been submitted to 2023 7th International Conference on Communication and Information Systems (ICCIS 2023) held in Chongqing, China.}
}
\author{\IEEEauthorblockN{Sixi Cheng and Xiang Ling} \\ 
\IEEEauthorblockA{\textit{National Key Laboratory of Science and Technology on Communications} \\
\textit{University of Electronic Science and Technology of China}\\
Chengdu, China \\
Email: chengsixi@std.uestc.edu.cn, xiangling@uestc.edu.cn.}}
\vspace{-8mm}

\begin{document}
\maketitle

\begin{abstract}
This article investigates the transmission delay of a Low Earth Orbit (LEO) satellite communication system in a bent pipe structure. By employing a stochastic geometry framework, satellites are modeled as spherical binomial point processes (BPP). A suboptimal satellite relay selection strategy is proposed, which achieves optimal conditions through theoretical analysis and numerical exploration. We derive the distance distributions for the uplink and downlink links, and provide corresponding analytical expressions for the transmission delays.
\end{abstract}

\section{Introduction}

Ultra-dense LEO satellite network, because of its seamless global coverage characteristics, is likely to be utilized as a vital part of the future 6G system \cite{de2021survey,kodheli2020satellite}.
Companies such as SpaceX, Telesat, and OneWeb are accelerating the formation of a network of tens of thousands of LEO satellites \cite{sheetz2019next}. 
In real time communication scenarios, satellites fundamentally play the role of a space relay or forwarder to connect two terrestrial stations \cite{ma2022secure}, which can be considered as a terrestrial-satellite-terrestrial unit in long-distance transmission. This leads to a fact that the performance of a terrestrial-satellite-terrestrial unit is worthy to analyse.

Stochastic geometry, as an effective mathematical tool, plays a particularly important role in analyzing the performance of satellite networks \cite{wang2022ultra,tian2023satellite}. BPP model, which is relatively accurate for closed area networks with a fixed number of satellites, is used in \cite{9079921} to analyse the coverage and rate of downlink. The user coverage probability for a scenario where satellite gateways (GWs) are deployed on the ground to act as a relay between the users and the LEO satellites is studied in \cite{talgat2020stochastic}.
Most of the authors of the focus on the downlink transmission performance of satellite-terrestrial, while the uplink transmission performance has not been extensively studied. \cite{wang2022conditional} introduces contact angle distribution to analyze the influences of the number of satellites and the distance between the transmitter and receiver, which does not consider any channel fading \cite{wang2023reliability}. As the Rician fading model is ubiquitous for the communication links between satellites and ground stations, we adopt shadowed-Rician (SR) fading for the channel between the satellite and the terrestrial station, which is pointed as most accurate channel model.

As for relay selection strategy, \cite{belbase2018coverage} proposes nearly optimal protocol in dual hop scenario. Therefore, based on the existing research, the contributions of this work are summarized as follows.

\begin{itemize}
    \item We give a possible optimal relay selection strategy and explore under what conditions the relay selection strategy is approximately optimal.
    \item Under the certain relay selection strategy, we derive analytical expressions of the distance distribution of uplink and downlink and respectively give expressions for the cumulative distribution function of the signal-to-noise ratio (SNR), which take channel fading into consideration.
    \item Based on above, we derive analytical expressions of the total transmission delay. By simulation, we verify the accuracy of the total transmission delay and investigate the effect of power, number of satellites, and distance between transmitter and receiver on the total transmission delay.
    \item In the simulation, stochastic geometry is used to analyze the performance of satellite communications over multiple hops (or multiple links), which is not found in the existing papers.
\end{itemize}


\begin{figure}[t]
  \centering
  \includegraphics[width=0.8\linewidth]{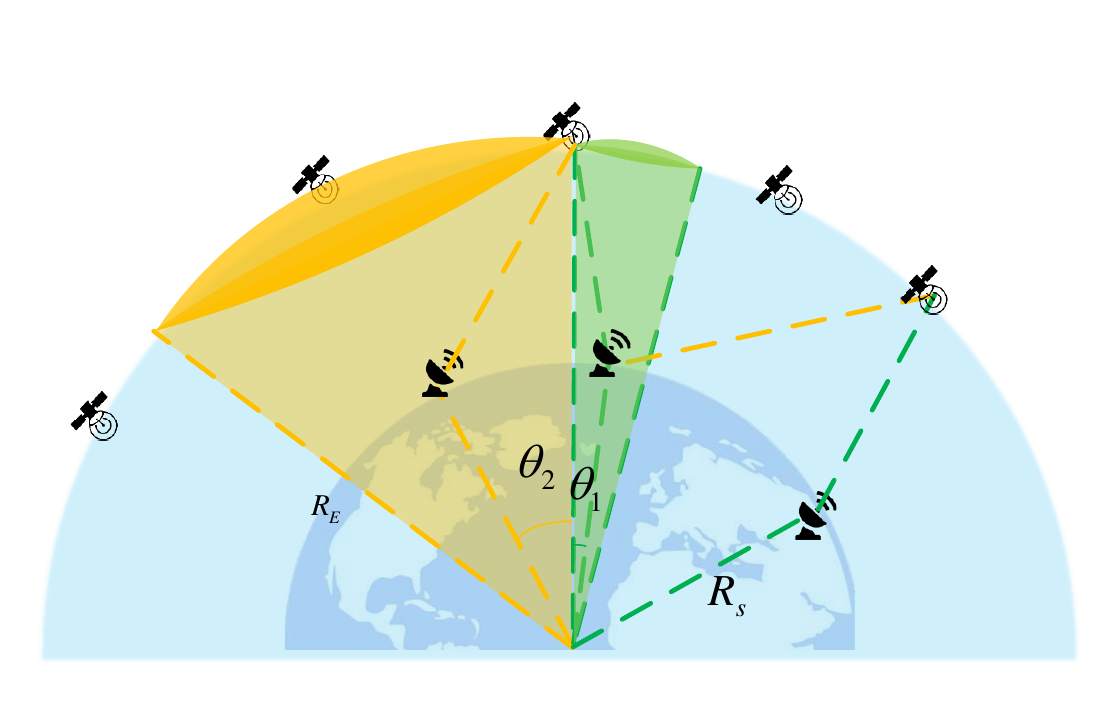}
  \caption{System model.}
  \label{sliding_window}
\end{figure}

\section{System Model}
\subsection{Network Topology}
In this subsection, we build a terrestrial-satellite-terrestrial relay communication model. $N_{s}$ satellites are distributed on a spherical surface with radius $R_s$ and form a homogeneous BPP \cite{wang2022evaluating}. Since the BPP distribution remains the same after the rotation, the coordinates of the transmitter and receiver are set to at ($R_e$,0,0) and ($R_e$,$\Theta$,0) for the convenience of calculation.
The radius of the Earth is denoted as $R_e$ = 6371km.
Here, we consider that the terrestrial stations are fixed and the satellites obey the BPP distribution.

Although we analyze only one terrestrial-satellite-terrestrial unit here, in the simulation part we analyse the long-distance transmission with multi units. This implies that the unit model is also applicable to the long-distance transmission models.

\subsection{Relay Selection}
In fixed topology Amplify and Forward (AF) or Decode and Forward (DF) networks, the optimal relay selection criterion is the maximization of end-to-end SNR or the maximum SNR of downlink.
However, when the relay satellites form a BPP, analysis of this strategy is intractable. Therefore, we consider a slightly suboptimal but tractable selection strategy \cite{lou2023coverage}: First, find a set of relays which can provide reliable communication for both the transmitter and receiver. Then, select the relay that has the strongest average received power in downlink. 
The reason for choosing the the strongest average received power in downlink is that: (i) In AF, the relay satellite amplifies the signal and also amplifies the noise. As a result, the downlink noise must be larger than the uplink. (ii) Considering that satellite energy is expensive, satellite transmission power is relatively low. 


\subsection{Channel Model}

Many works have focused on deriving an accurate channel model for the communication links between satellites and ground stations, where it was shown that shadowed-Rician (SR) model proposed in \cite{1198102} is the most accurate. So, we consider the SR model to calculate fade margins and analyze the performance of communication.

For the uplink, the received signal power $p_{(2)}^s$ at relay satellite is
\begin{equation}
p_{(2)}^s = p_{(1)}^sA_u \times W_{t}^2,
\end{equation}
where $A_u$ and $W_{t}^2$ respectively represent the propagation loss and the SR fading. The propagation loss can be calculated by:
\begin{equation}
A_u = \frac{G_{(1)}^e G_{(2)}^s \lambda_u^2}{(4 \pi d_{\mathrm{up}})^2L_{(1)}^e L_{(2)}^s L_{\mathrm{add}}},
\end{equation}
where $G_{(1)}^e$ and $G_{(2)}^s$ denote the transmitter and receiver antenna gain, $\lambda_u$ denotes the carrier wavelength of uplink, $d_{\mathrm{up}}$ is the distance between the transmitter and the relay satellite, $L_{(1)}^e$ and $L_{(2)}^s$ denote the transmitter and receiver antenna feeder loss, $L_{\mathrm{add}}$ denotes the link additional loss, including atmospheric absorption loss, rain attenuation, etc.

The probability density function (PDF) of the SR fading power \cite{1198102} $W_t^2$ is given as follows:
\begin{equation}
\begin{split}
f_{W_t^2}(t) &= \left(\frac{2b_0m}{2b_0m+\Omega}\right)^m \frac{1}{2b_0} \exp \left(-\frac{t}{2b_0}\right) 
\\
& \times \sum\limits_{n=0}^{\infty }\frac{(m)_n}{(1)_n n!}{{ \left(\frac{\Omega \,t}{2b_0(2b_0m+\Omega )} \right)}^{n}}, \ \ \ \ t \geq 0,
\end{split}
\end{equation}
where $(\cdot)_{n}$ is the Pochhammer symbol, while $m$, $b_0$ and $\Omega$ are the parameters of the SR fading.
With the system model above, the received SNR for a link is given by
\begin{equation}
\label{4}
{\mathrm{SNR}}_\mathrm{up} = \frac{p_{(2)}^s}{N_u}
= \frac{p_{(1)}^eG_{(1)}^e G_{(2)}^s \lambda_u^2W_t^2}{(4 \pi d_{\mathrm{up}})^2L_{(1)}^e L_{(2)}^s L_{\mathrm{add}}N_u}
, 
\end{equation}
where $N_u$ is the noise power of uplink. $N_u$ can be calculated by $N_u = kBT_{u}$, where $k$ is the Boltzmann's constant, $B$ is the bandwidth of transmission and $T_{u}$ denotes the total network noise temperature of the uplink.

The received SNR for a link is given by
\begin{equation}
\label{5}
{\mathrm{SNR}}_\mathrm{down} = \frac{p_{(2)}^e}{N_d}
= \frac{p_{(1)}^sG_{(1)}^s G_{(2)}^e \lambda_d^2W_t^2}{(4 \pi d_{\mathrm{down}})^2L_{(1)}^s L_{(2)}^e L_{\mathrm{add}}N_d},
\end{equation}
where $N_d$ is the noise power of downlink and $p_{(2)}^e$ can be calculated as the same as $p_{(2)}^s$ with downlink parameters.


\section{Time delay analysis}

\subsection{Distance Distribution}


In order to contribute expressions for average time delay in the following sections, we first need to characterize some basic distance distributions that stem from the stochastic geometry of the considered system.
According to the relay selection strategy mentioned above, the downlink distance $d_{\mathrm{down}}$ is a independent from $\theta_2$.

Since the correspondence between the central angle (the angle of the line from two points to the center of the Earth \cite{wang2022stochastic}) and distance is bijective, the distance distribution can be calculated via calculating the distribution of the central angle (the angle of the line between two points and the center of the Earth).

To facilitate the calculation of the uplink distance distribution, we use the CDF of the central angle. The central angle of the receiver and the relay satellite is denoted as $\theta_1$, and the central angle of the transmitter and the relay satellite is denoted as $\theta_2$.

\begin{lemma}\label{lemma1}
The PDF of the $\theta_1$ from any specific one of the satellites in the constellation to the receiver is given by
\begin{equation}\label{fpsinn}
\begin{split}
    f_{\theta_1}(\theta) = \frac{N_s \sin\theta}{2} \left( \frac{ 1 + \cos\theta }{2} \right)^{N_s-1},0\leq \theta \leq \pi.
\end{split}
\end{equation}
\begin{proof}
See Appendix~\ref{app:lemma1}.
\end{proof}
\end{lemma}

\begin{lemma}\label{lemma2}
The downlink distance distribution $d_{\mathrm{up}}$ is given by
\begin{equation}\label{eq2}
\begin{split}
	F_{d_{\mathrm{down}}}(d_0)
 &=\left\{
	\begin{aligned}
            &0  , d_0 \leq R_s-R_e\\
            & 1-\left(1-\left( \frac{d_0^2-(R_s-R_e)^2}{4R_eR_s} \right)\right)^{ N_{s} }, 
            \mathrm{otherwise}\\
		&1  ,d_0 \geq R_s+R_e\\	
	\end{aligned}
	\right.
 \end{split}
\end{equation}

\begin{proof}
See Appendix~\ref{app:lemma2}.
\end{proof}
\end{lemma}

Due to the relay selection strategy, $\theta_1$ is an independent variable while $\theta_2$ is a random variable associated with $\theta_1$.

It is important to note that we consider the probability of a satellite appearing in a circle ring with a fixed central angle $\theta_1$ to be uniformly distributed. However, the probability of a satellite appearing in a circle ring with a fixed central angle $\theta_2$ is weighted by $f_{\theta_1}(\theta)$.


\begin{lemma}\label{lemma3}
Given that the maximum central angle of the transmitter’s spherical cap is $\beta$, the approximate CDF of $ F_{\theta_2}(\beta)$ is given by
\begin{equation}
    \begin{split}
    F_{\theta_2} \left(\beta\right) = \int_0^{2\pi} \int_0^{\beta} \frac{f_{\theta_1}\left( \psi(\theta,\phi) \right)}{2\pi R_s^2 \sin\psi(\theta,\phi)}  R_s^2 \sin\theta \mathrm{d}\theta \mathrm{d}\phi,
\end{split}
\end{equation}
where $\psi(\theta,\phi)$ is represented by
\begin{equation}
\begin{split}
    &\psi(\theta,\phi) = 2\arccos \frac{R_s^2+R_e^2-\mathcal{D}^2(\theta,\phi,\Theta,0)}{2R_sR_e},
\end{split}
\end{equation}
and $\Phi$ is the central angle between the transmitter and receiver. ${\mathcal{D}^2(\theta_1,\phi_1,\theta_2,\phi_2)}$  is an operator and can be expressed by

\begin{equation}
\begin{split}
&{\mathcal{D}^2(\theta_1,\phi_1,\theta_2,\phi_2)}={R_e^2}+{R_s^2}
\\
&-2{R_eR_s}(1-(\cos {{\phi }_{1}}\cos {{\phi }_{2}}\cos ({{\theta }_{1}}-{{\theta }_{2}})+\sin {{\theta }_{1}}\sin {{\theta }_{2}))}.
\end{split}
\end{equation}

\begin{proof}
See Appendix~\ref{app:lemma3}.
\end{proof}
\end{lemma}

Above all, we can obtain the uplink distance distribution $d_{\mathrm{up}}$ in the following lemma.
\begin{lemma}
The uplink distance distribution $d_{\mathrm{up}}$ is given by

\begin{equation}
\label{Fup}
\begin{split}
	F_{d_{\mathrm{up}}}(d_0)
 &=\left\{
	\begin{aligned}
            &0  , d_0 \leq R_s-R_e\\
            &\frac{d_0}{R_sR_s\sqrt{1-\left(\frac{R_s^2+R_e^2-d_0^2}{2(R_s+R_e)}\right)^2}}\\
            &\times F_{\theta_2} \left(\mathrm{arccos}\left(\frac{R_s^2+R_e^2-d_0^2}{2(R_s+R_e)}\right)\right),
            \mathrm{else}\\
		&1, d_0 \geq R_s+R_e\\	
	\end{aligned}
	\right.
 \end{split}
\end{equation}
\end{lemma}
\begin{proof}
Since the expressions for the central angle $\theta_2$ and downlink distance $d_\mathrm{up}$ is given by
\begin{equation}
\begin{split}
d_{\mathrm{up}} = \sqrt{R_s^2+R_e^2-2R_sR_e\cos\theta_2}\\
\theta_2 = \mathrm{arccos}\left(\frac{R_s^2+R_e^2-d_{\mathrm{up}}^2}{2R_sR_e}\right),
\end{split}
\end{equation}
and the uplink distance distribution $d_{\mathrm{up}}$ is given by (\ref{Fup}).

\end{proof}

\subsection{Time Delay}
We define the transmission time delay of a link as
\begin{equation}
\tau_{Q} = \frac{M}{B\log_2({1+\mathrm{SNR}_{Q}})},Q\in \{\mathrm{up},\mathrm{down}\},
\end{equation}
where $M$ and $B$ respectively denote the size of the packet and the bandwidth of transmission. Thus we can derive the total time delay $\tau_{\mathrm{total}}$ in following theorem.


\begin{theorem}\label{theorem1}
The average time delay of downlink $\overline{\tau}_{\mathrm{down}} $ is given by
\begin{equation}
\overline{\tau}_{\mathrm{down}}=\frac{M}{B\log_2\left({1+\frac{p_{(1)}^sG_{(1)}^s G_{(2)}^e \lambda_d^2(\Omega+2b_0))}{(4 \pi \overline{d}_{\mathrm{down}})^2L_{(1)}^s L_{(2)}^e L_{\mathrm{add}}N_d}}\right)}.
\end{equation}

\begin{proof}
See Appendix~\ref{app:theorem1}.
\end{proof}
\end{theorem}

\begin{theorem}\label{theorem2}
The average time delay of uplink $\overline{\tau}_{\mathrm{up}} $ is given by
\begin{equation}
\overline{\tau}_{\mathrm{up}} = \int_{0}^{\infty} \frac{M}{B\log_2({1+\gamma})} f_{{\mathrm{SNR}}_\mathrm{up}}(\gamma) \mathrm{d}\gamma,
\end{equation}
where 
\begin{equation}
\begin{split}
    f_{{\mathrm{SNR}}_\mathrm{up}}(\gamma) &= \int_{(R_s-R_e)^2}^{(R_s+R_e)^2} \frac{\sqrt{u}}{2} \left(\frac{2b_0m}{2b_0m+\Omega}\right)^m  \exp \left(-\frac{z(\gamma)u}{2b_0}\right)\\
    & \times \frac{1}{2b_0} \sum\limits_{n=0}^{\infty }\frac{(m)_n}{(1)_n n!}{{ \left(\frac{\Omega \,z(\gamma)u}{2b_0(2b_0m+\Omega )} \right)}^{n}} f_{d_{\mathrm{up}}}(\sqrt{u})\mathrm{d}u,
\end{split}
\end{equation}
and $ z(\gamma) = \frac{(4 \pi)^2L_{(1)}^e L_{(2)}^s L_{\mathrm{add}}N_u\gamma}{p_{(1)}^eG_{(1)}^e G_{(2)}^s \lambda_u^2}$.
\begin{proof}
See Appendix~\ref{app:theorem2}.
\end{proof}
\end{theorem}

\subsection{Optimality Analysis}
Due to the relay selection strategy and earth blockage, the transmitter and receiver can only communicate with satellites within a maximum distance $L_{\mathrm{max}}$:
\begin{equation}
\label{Lmax}
L_{\mathrm{max}} = 2R_e\sin\left(\frac{1}{2}\arccos\left(\frac{R_e}{R_s}\right)\right).
\end{equation}

Under the premise of satisfying the above inequalities, we solve the optimization problem below and explore the optimal relay position of the satellite through numerical results.
\begin{equation} 
 \theta^* = \underset{0 \leq \theta \leq \Theta}{\mathrm{argmin}} \ \overline{\tau}_{\mathrm{up}} + \overline{\tau}_{\mathrm{down}}.
\end{equation}

\begin{figure}[H]
  \centering
  \includegraphics[width=0.7\linewidth]{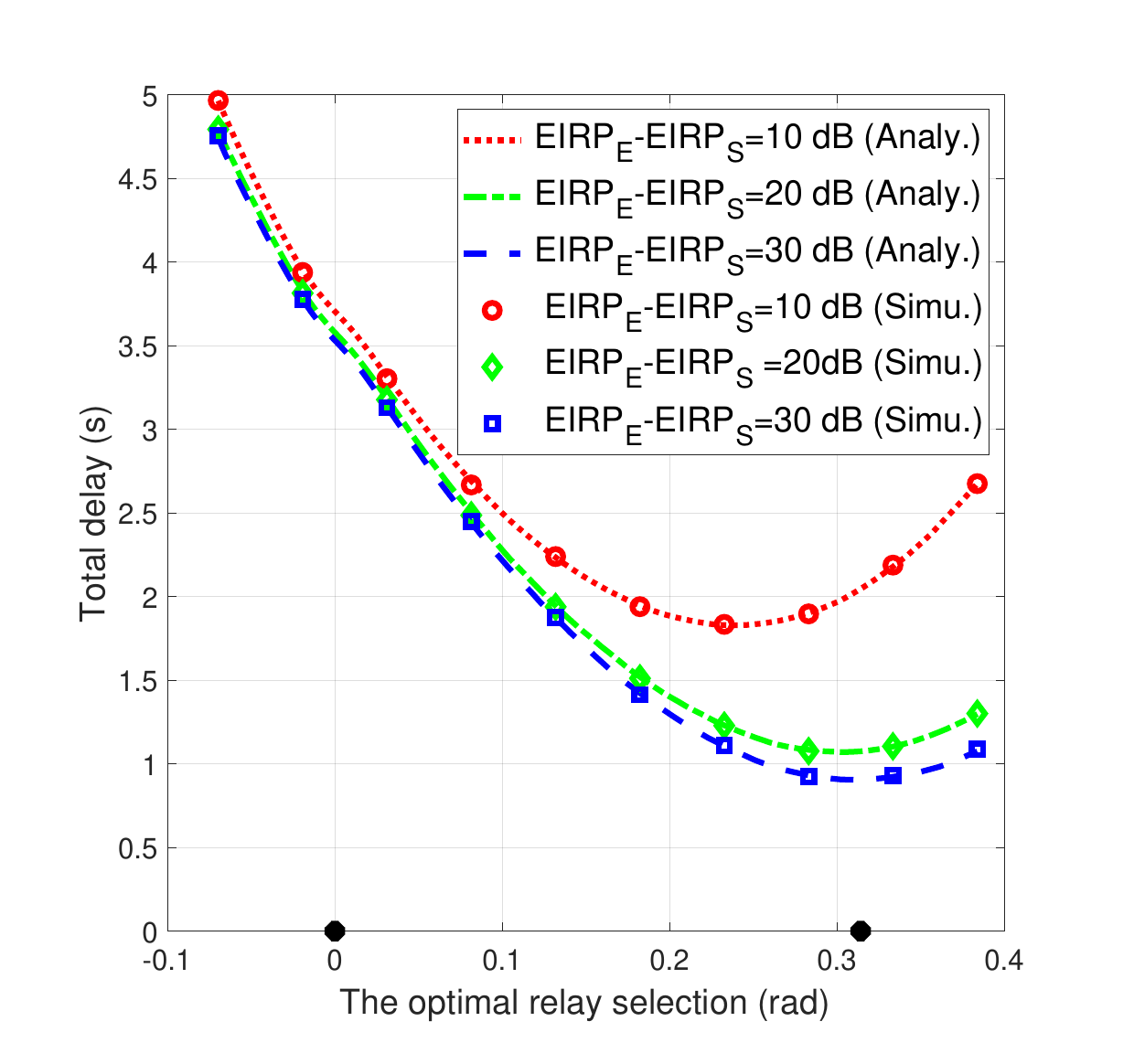}
  \caption{Total time delay with different relay selection.}
  \label{simufigure1}
\end{figure}

In Fig.\ref{simufigure1}, we use markers to denote the polar angle coordinate position of the terrestrial station. The horizontal coordinate is the polar angle of the ideal relay satellite. In the Monte-Carlo simulation, we search for the closest satellite to the ideal satellite coordinates as a relay satellite to calculate the uplink and downlink time delays. It can be seen that choosing the closest satellite to the receiver is a good suboptimal strategy when the ratio of uplink and downlink power differs significantly.

\section{Numerical Results}

\begin{table*}[htbp]
  \begin{center}
    \caption{Table of Notations.}
    \label{table_1}
    \begin{tabular}{c|c|c}
    \hline
      \textbf{\textit{Notations}}&\textbf{\textit{Description}}&\textbf{\textit{Value(Default)}}\\
      \hline
      $N_s$ & Number of satellites & 500\\ \hline
      $R_e$;\,$R_s$ & Radius of the Earth; satellites & 6371; 6871 (km)\\ \hline
      $N_u$;\,$N_d$ & Noise power of uplink; downlink & $3.6\times10^{-12}$; $3.6\times10^{-12}$\\ \hline
      $B_u$;\,$B_d$ & Bandwidth of uplink; downlink & $0.5$; $0.25$ (GHz)\\ \hline
      $p_{(1)}^sG_{(1)}^s$;\,$p_{(1)}^eG_{(1)}^e$ & Effective isotropic radiated power at relay satellite; transmitter & 30; 60 (dB)\\\hline
       $L_{\mathrm{add}}$ & Link additional loss & 3 (dB)\\\hline
      $\lambda_d$;\,$\lambda_u$ & Carrier wavelength of downlink; uplink & 0.0231; 0.015 (m)\\\hline
      $\Omega$;\,$b_0$;\,$m$ & Line-of-sight component; scatter component; Nakagami parameter & 1.29; 0.158; 19.4\\\hline
      $M$ & The size of the packet & 0.5 Gbit\\\hline
      \hline
    \end{tabular}
  \end{center}
\end{table*}

In this section, we verify the accurancy of the derived expressions using Monte-Carlo simulations. In addition, we study the influence of various system parameters on the performance of the considered system. In all the figures, markers represent the derived analytical results while the solid lines represent the Monte-Carlo simulations. The system parameters used in the simulations are summarized in Table \ref{table_1}.
\begin{figure}[t]
  \centering
  \includegraphics[width=0.65\linewidth]{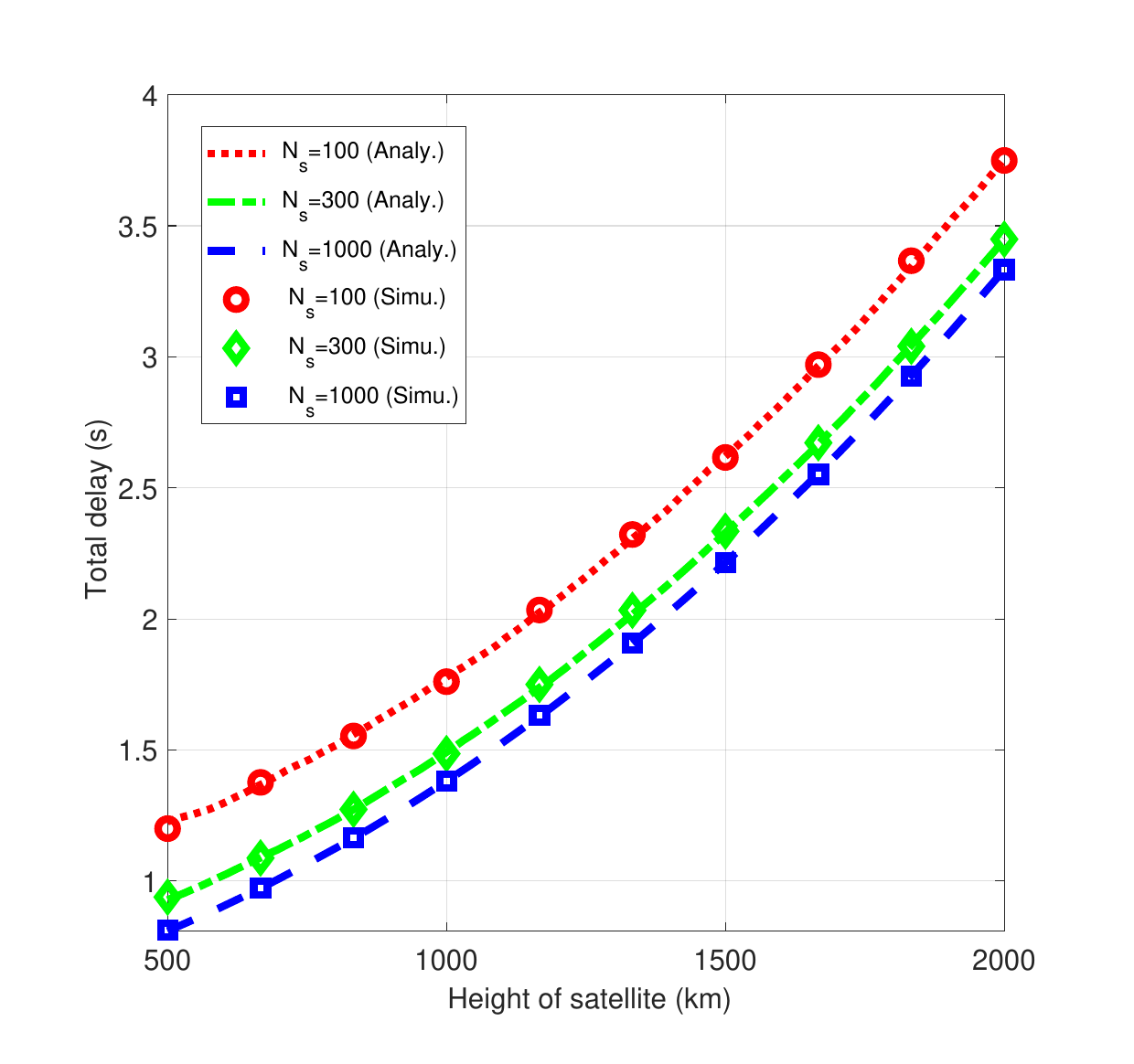}
  \caption{Total time delay with different number of satellite.}
  \label{simufigure2}
\end{figure}

\begin{figure}[t]
  \centering
  \includegraphics[width=0.65\linewidth]{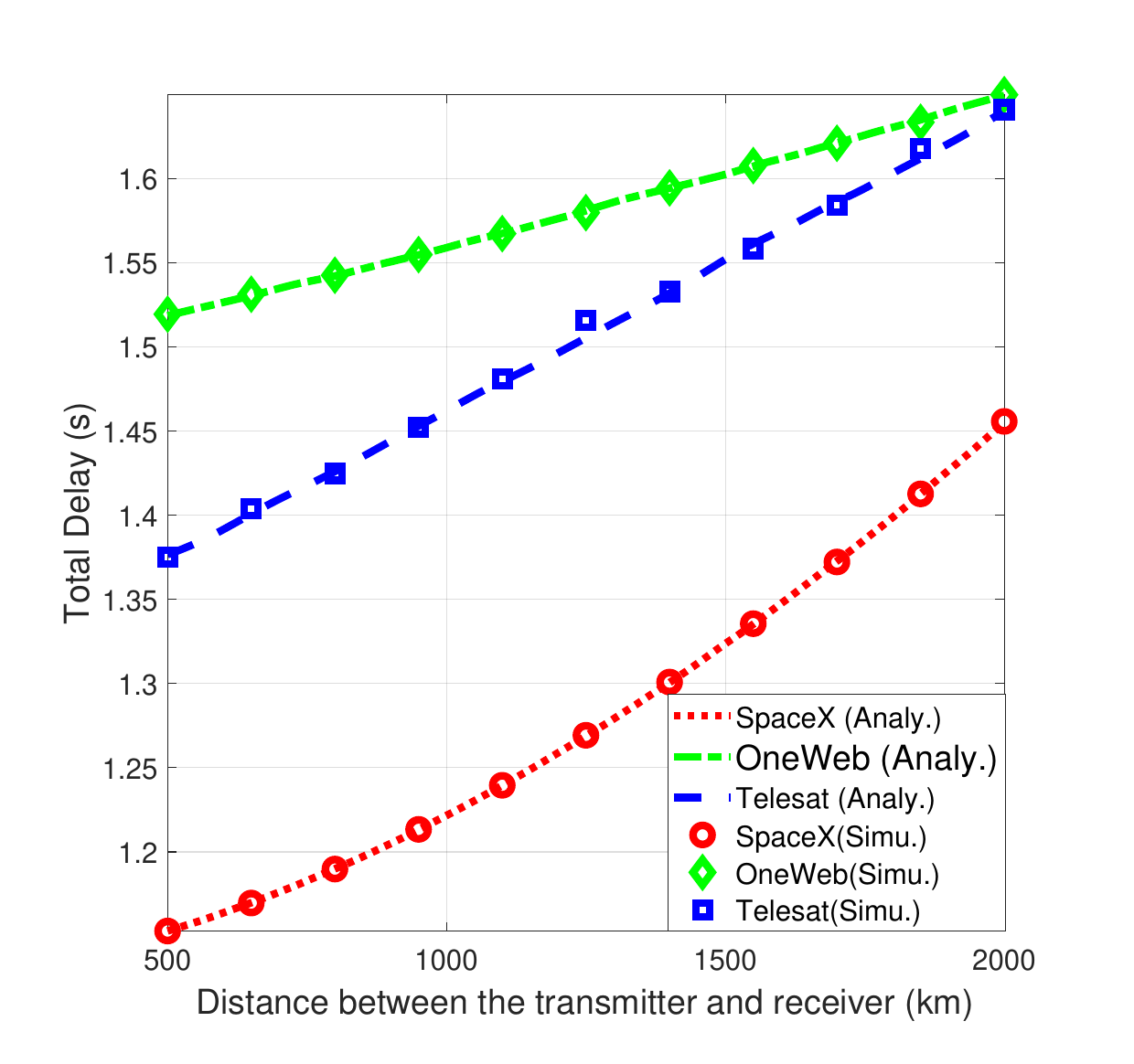}
  \caption{Total time delay with different distance between the transmitter and receiver.}
  \label{simufigure3}
\end{figure}

\begin{figure}[t]
  \centering
  \includegraphics[width=0.65\linewidth]{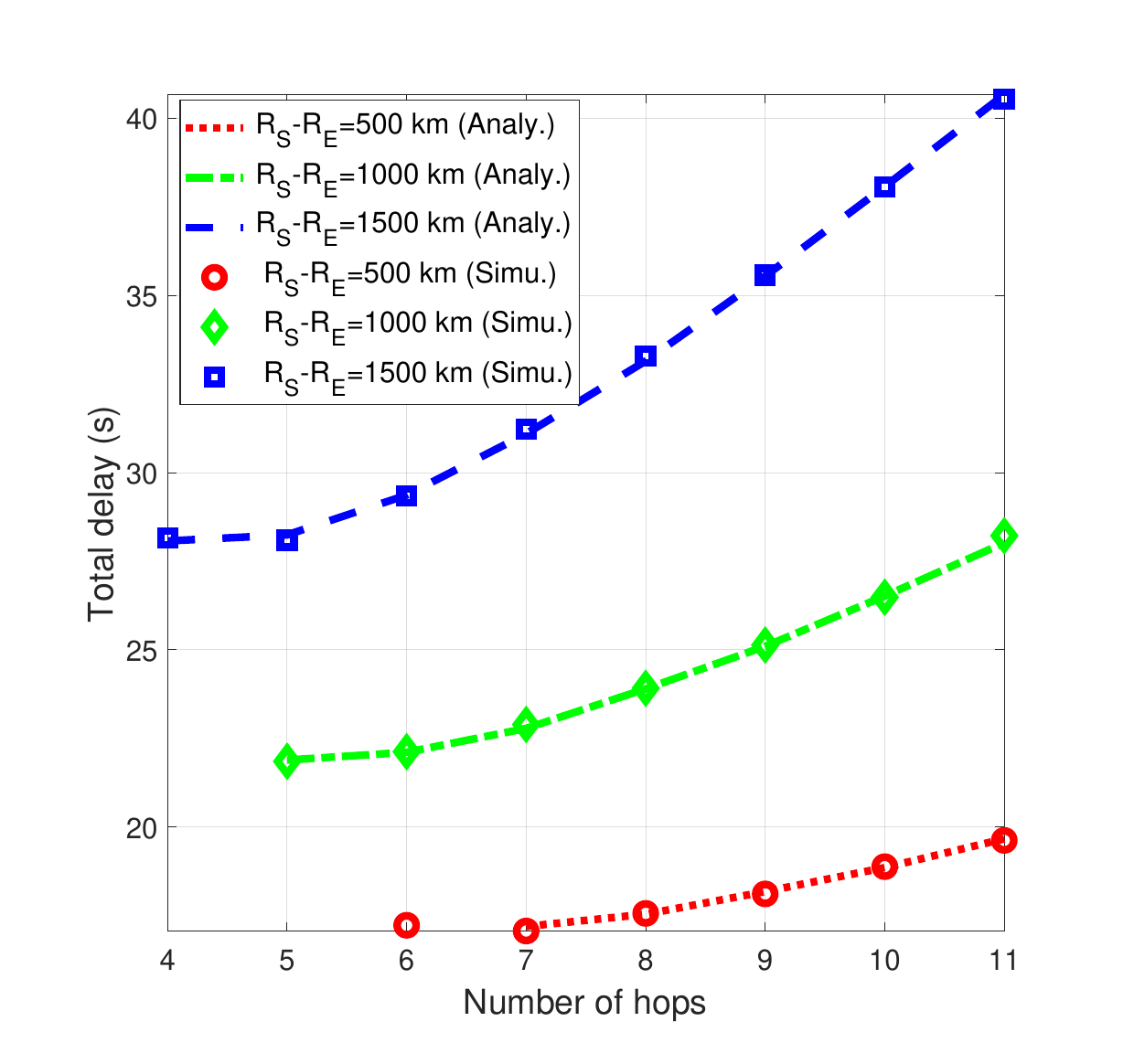}
  \caption{Total time delay with different hops in fixed long-distance transmission.}
  \label{simufigure4}
\end{figure}

In Fig.\ref{simufigure2}, we plot total delay for different fixed number of satellites and study the effect of increasing the altitude of satellite. The results show that transmission delay increase as the the height of satellite and we can observe that for a fixed altitude time delay reduces as we increase the number of satellite.

In Fig.\ref{simufigure3}, the time delay is studied under different values for the distance between the transmitter and receiver for companies OneWeb, Telesat and SpaceX with altitudes $h = 1200$ km, $h = 1150$ km and $h = 1110$  km \cite{del2019technical}. Other parameters such as the number of satellites, effective isotropic radiated power, carrier frequency, bandwidth, etc. are referenced in \cite{del2019technical}. 

Fig.\ref{simufigure4} illustrates the scenario of a long distance transmission containing multiple hops. We set the total transmission distance to 15000 km and calculate the minimum number of hops required at different altitude satellites according to (\ref{Lmax}). As the altitude of the satellite increases, it is less affected by ground obscuration and the maximum distance between both sides of communication on the ground is increasing.

\section{Conclusion}
In this work, we propose a suboptimal satellite relay selection strategy in terrestrial-satellite-terrestrial scenario. Though deriving theoretical expressions of the downlink and uplink distance distribution, we give a expression of total transmission delay. We have verified all the derived expressions using Monte-Carlo simulations and ensured perfect fit. In simulation, we explore the conditions for the proposed strategy to reach optimal and provide the numerical results about the influence of the altitudes of the satellites, their numbers, the distance between two terrestrial stations on the performance of the delay.

\section{Acknowledgement}
This work was supported by the National Key Research and Development Program of China (No. 2021YFB2900404).

\appendices
\section{Proof of Lemma~\ref{lemma1}}\label{app:lemma1}
For a homogeneous BPP, the probability of the satellite locates in a spherical cap with central angle $\theta$ is equal to the ratio of the area of $\theta$ to the total surface area of the sphere. So, we can obtain 
\begin{equation}
\begin{split}
    \mathbb{P}\left[  \theta_1 \leq \theta\right] &= \frac{\int_0^{2\pi} \int_0^{\theta} R_s sin(\theta_1) \mathrm{d}\theta_1 \mathrm{d}\phi}{4\pi R_s^2}\\
    &= \frac{ 2\pi R_s^2 (1-\cos\theta)}{4\pi R_s^2},0\leq \theta \leq \pi. \\
\end{split}
\end{equation}

Due to the channel assignment by which the serving satellite is the nearest one among all the $N_s$ i.i.d. satellites, the CDF of the $\theta_1$ from any specific one of the satellites in the constellation to the receiver is given by 
\begin{equation}
\begin{split}
    F_{\theta_1}(\theta) 
    & = 1 - \prod_{i=1}^{N_s} \mathbb{P}\left[  \theta_1 \geq \theta\right] \\
    & = 1 - \left( 1 - \frac{ 2\pi R_s^2 (1-\cos\theta)}{4\pi R_s^2} \right)^{N_s}\\
    & = 1 - \left( \frac{ 1 + \cos\theta }{2} \right)^{N_s},\\
\end{split}
\end{equation}
The PDF of $\theta_1$ is
\begin{equation}\label{fpsinn}
\begin{split}
    f_{\theta_1}(\theta) = \frac{\mathrm{d}}{\mathrm{d}\theta} F_{\theta_1}(\theta) = \frac{N_s \sin\theta}{2} \left( \frac{ 1 + \cos\theta }{2} \right)^{N_s-1},
\end{split}
\end{equation}
where $ \frac{\mathrm{d}}{\mathrm{d}\theta}$ means take the derivative with respect to $\theta$.        

\section{Proof of Lemma~\ref{lemma2}}\label{app:lemma2}

In the authors’ work \cite{9079921}, the PDF of $d_{\mathrm{down}}$ was derived as shown in the below lemma. 

Due to the channel assignment by which the serving satellite is the nearest one among all the $N_s$ i.i.d. satellites, the PDF of the $d_{\mathrm{down}}$ from any specific one of the satellites in the constellation to the user is given by
\begin{equation}
\label{fdown}
f_{d_{\mathrm{down}}}(d_0) = N_{s} \left( 1- \frac{d_0^2-(R_s-R_e)^2}{4R_eR_s}\right)^{ N_{s}-1 } \frac{d_0}{2R_eR_s},
\end{equation}
for $R_s-R_e\leq d_0 \leq R_s+R_e$ while $f_{d_{\mathrm{down}}}(d_0)=0$ otherwise.
The CDF can be expressed as 
\begin{equation}
\begin{split}
	F_{d_{\mathrm{down}}}(d_0)
 &=\left\{
	\begin{aligned}
            &0  , d_0 \leq R_s-R_e\\
            & 1-\left(1-\left( \frac{d_0^2-(R_s-R_e)^2}{4R_eR_s} \right)\right)^{ N_{s} }, 
            \mathrm{else}\\
		&1  ,d_0 \geq R_s+R_e\\	
	\end{aligned}
	\right.
 \end{split}
\end{equation}

\section{Proof of Lemma~\ref{lemma3}}\label{app:lemma3}

\begin{figure}[t]
  \centering
  \includegraphics[width=\linewidth]{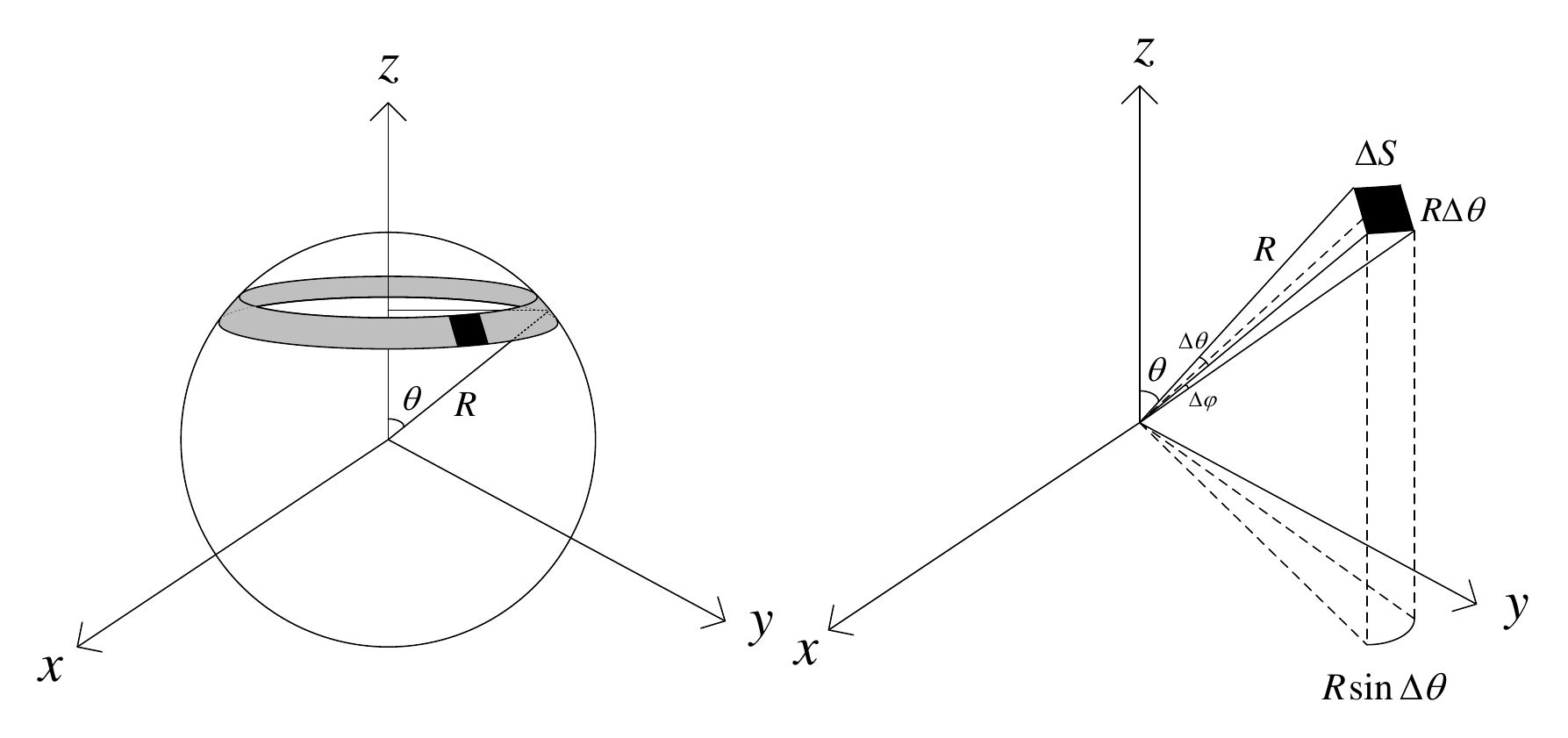}
  \caption{Illustrations of Lemma 3 and Lemma 4.}
  \label{figure2}
\end{figure}

To derive the distribution of $\theta_2$, the following steps are taken: (i): assuming that the coordinate position of the satellite is $(R_s,\theta_0,\phi_0)$; (ii): in order to calculate the probability of the satellite appearing in the spherical cap (corresponds to the central angle $\beta$), perform a double integration over the probability of the satellite with respect to $\theta_0$ and $\phi_0$, where $0 \leq \theta_0 \leq \beta$ and $0 \leq \phi_0 \leq 2\pi$; (iii): the probability of occurrence at $(R_s,\theta_0,\phi_0)$ is weighted by $\frac{f_{\theta_1}\left( \psi(\theta,\phi) \right)}{2\pi R_s^2 \sin\psi(\theta,\phi)}$.

Given that the maximum central angle of the transmitter’s spherical cap is $2\beta$, the approximate CDF of $ F_{\theta_2}(\beta)$ is given by
\begin{equation}
    \begin{split}
    F_{\theta_2} \left(\beta\right) = \int_0^{2\pi} \int_0^{\beta} \frac{f_{\theta_1}\left( \psi(\theta,\phi) \right)}{2\pi R_s^2 \sin\psi(\theta,\phi)}  R_s^2 \sin\theta \mathrm{d}\theta \mathrm{d}\phi.
\end{split}
\end{equation}
And $\psi(\theta,\phi)$ is represented by
\begin{equation}
\psi(\theta,\phi) = 2\arccos \left(\sin\theta\sin\Theta\cos\phi+\cos\theta\cos\Theta\right),
\end{equation}
where $\Theta$ is the central angle between the transmitter and receiver.

\section{Proof of Theorem~\ref{theorem1}}\label{app:theorem1}

The average time delay of uplink $\overline{\tau}_{\mathrm{down}}$ is given by
\begin{equation}
\label{27}
\overline{\tau}_{\mathrm{down}} = \frac{M}{B\log_2({1+\overline{\mathrm{SNR}}_{\mathrm{down}}})},
\end{equation}
where $\mathbb{E}[\cdot]$ denotes taking the mathematical expectation.

The average central angle $\overline{\theta}_1$ is given by:
\begin{equation}
\begin{split}
\label{avertheta1}
\overline{\theta}_1 &= \mathbb{E}[\theta_1] = \int_{0}^{\pi}1-F_{\theta_1}(\theta) \mathrm{d}\theta\\
&=\int_{0}^{\pi}\left( \frac{ 1 + \cos\theta }{2} \right)^{N_s}\mathrm{d}\theta =\int_{0}^{\pi} (\cos\frac{\theta}{2})^{2N_s} \mathrm{d}\theta\\
&=2\int_{0}^{\frac{\pi}{2}} (\cos\theta)^{2N_s}\mathrm{d}\theta=\pi\prod_{n=0}^{N_s-1}\frac{2N_s-2n-1}{2N_s-2n}
\end{split}
\end{equation}

The average distance of downlink $\overline{d}_{\mathrm{down}}$ can be calculated by:
\begin{equation}
\label{averdowndis}
\overline{d}_{\mathrm{down}} = \sqrt{R_e^2+R_s^2-2R_eR_s\cos\overline{\theta}_1}.
\end{equation}

According to \cite{1198102}, the moment-generating function (MGF) of the instantaneous power can be shown to be
\begin{equation}
M_{W_t^2}(\sigma )= \frac{(2b_0m)^m(1+2b_0\sigma)^{m-1}}{[(2b_0m+\Omega)(1+2b_0\sigma)-\Omega]^m},\sigma \geq 0.
\end{equation}

The mathematical expectation of $W_t^2$ is given by:
\begin{equation}
\label{SRexp}
\begin{split}
\mathbb{E}[W_t^2] = \frac{\mathrm{d}}{\mathrm{d}\sigma}M_{W_t^2}(0) = \Omega+2b_0.
\end{split}
\end{equation}

Recall (\ref{5}), by substituting (\ref{avertheta1}, (\ref{averdowndis}) and (\ref{SRexp}) into (\ref{27}), we obtain the average time delay of downlink $\overline{\tau}_{\mathrm{down}}$:
\begin{equation}
\overline{\tau}_{\mathrm{down}}=\frac{M}{B\log_2\left({1+\frac{p_{(1)}^sG_{(1)}^s G_{(2)}^e \lambda_d^2(\Omega+2b_0))}{(4 \pi \overline{d}_{\mathrm{down}})^2L_{(1)}^s L_{(2)}^e L_{\mathrm{add}}N_d}}\right)}.
\end{equation}

\section{Proof of Theorem~\ref{theorem2}}\label{app:theorem2}
Recall (\ref{4}), since $W_t^2$ and $d_{\mathrm{down}}$ are independent random variables, the PDF of ${\mathrm{SNR}}_\mathrm{up}$ is a two-dimensional random variable and can be given by
\begin{equation}
\begin{split}
&f_{{\mathrm{SNR}}_\mathrm{up}}(\gamma)\\
&= \mathbb{P}\left(\frac{p_{(1)}^eG_{(1)}^e G_{(2)}^s \lambda_u^2W_t^2}{(4 \pi d_{\mathrm{up}})^2L_{(1)}^e L_{(2)}^s L_{\mathrm{add}}N_u} = \gamma \right)\\
&= \int_{(R_s-R_e)^2}^{(R_s+R_e)^2} u f_{W_t^2}\left(z(\gamma)u\right)f_{d_{\mathrm{down}}^2}(u)\mathrm{d}u\ \\
&= \int_{(R_s-R_e)^2}^{(R_s+R_e)^2} \frac{\sqrt{u}}{2}f_{W_t^2}\left(z(\gamma)u\right)f_{d_{\mathrm{down}}}(\sqrt{u})\mathrm{d}u\\
&= \int_{(R_s-R_e)^2}^{(R_s+R_e)^2} \frac{\sqrt{u}}{2} \left(\frac{2b_0m}{2b_0m+\Omega}\right)^m \frac{1}{2b_0} \exp \left(-\frac{z(\gamma)u}{2b_0}\right)\\
&\times \sum\limits_{n=0}^{\infty }\frac{(m)_n}{(1)_n n!}{{ \left(\frac{\Omega \,z(\gamma)u}{2b_0(2b_0m+\Omega )} \right)}^{n}} f_{d_{\mathrm{up}}}(\sqrt{u})\mathrm{d}u,\\
\end{split}
\end{equation}
where $ z(\gamma) = \frac{(4 \pi)^2L_{(1)}^e L_{(2)}^s L_{\mathrm{add}}N_u\gamma}{p_{(1)}^eG_{(1)}^e G_{(2)}^s \lambda_u^2}$.
\vspace{1mm}

The average time delay of uplink $\overline{\tau}_{\mathrm{up}}$:

\begin{equation}
\begin{split}
\overline{\tau}_{\mathrm{up}} &= \mathbb{E}\left[\frac{M}{B\log_2({1+\mathrm{SNR}_{\mathrm{up}}})} \right]\\
&= \int_{0}^{\infty} \frac{M}{B\log_2({1+\gamma})} f_{{\mathrm{SNR}}_\mathrm{up}}(\gamma) \mathrm{d}\gamma.
\end{split}
\end{equation}
\vspace{3mm}


\bibliographystyle{IEEEtran}
\bibliography{references}

\begin{thebibliography}{10}
\providecommand{\url}[1]{#1}
\csname url@samestyle\endcsname
\providecommand{\newblock}{\relax}
\providecommand{\bibinfo}[2]{#2}
\providecommand{\BIBentrySTDinterwordspacing}{\spaceskip=0pt\relax}
\providecommand{\BIBentryALTinterwordstretchfactor}{4}
\providecommand{\BIBentryALTinterwordspacing}{\spaceskip=\fontdimen2\font plus
\BIBentryALTinterwordstretchfactor\fontdimen3\font minus
  \fontdimen4\font\relax}
\providecommand{\BIBforeignlanguage}[2]{{%
\expandafter\ifx\csname l@#1\endcsname\relax
\typeout{** WARNING: IEEEtran.bst: No hyphenation pattern has been}%
\typeout{** loaded for the language `#1'. Using the pattern for}%
\typeout{** the default language instead.}%
\else
\language=\csname l@#1\endcsname
\fi
#2}}
\providecommand{\BIBdecl}{\relax}
\BIBdecl

\bibitem{de2021survey}
C.~De~Alwis, A.~Kalla, Q.-V. Pham, P.~Kumar, K.~Dev, W.-J. Hwang, and
  M.~Liyanage, ``Survey on 6g frontiers: Trends, applications, requirements,
  technologies and future research,'' \emph{IEEE Open Journal of the
  Communications Society}, vol.~2, pp. 836--886, 2021.

\bibitem{kodheli2020satellite}
O.~Kodheli, E.~Lagunas, N.~Maturo, S.~K. Sharma, B.~Shankar, J.~F.~M. Montoya,
  J.~C.~M. Duncan, D.~Spano, S.~Chatzinotas, S.~Kisseleff \emph{et~al.},
  ``Satellite communications in the new space era: A survey and future
  challenges,'' \emph{IEEE Communications Surveys \& Tutorials}, vol.~23,
  no.~1, pp. 70--109, 2020.

\bibitem{sheetz2019next}
M.~Sheetz and M.~Petrova, ``Why in the next decade companies will launch
  thousands more satellites than in all of history,'' \emph{CNBC. Last accessed
  October}, vol.~22, p. 2021, 2019.

\bibitem{ma2022secure}
Y.~Ma, T.~Lv, G.~Pan, Y.~Chen, and M.-S. Alouini, ``On secure uplink
  transmission in hybrid rf-fso cooperative satellite-aerial-terrestrial
  networks,'' \emph{IEEE Transactions on Communications}, vol.~70, no.~12, pp.
  8244--8257, 2022.

\bibitem{wang2022ultra}
R.~Wang, M.~A. Kishk, and M.-S. Alouini, ``Ultra-dense {LEO} satellite-based
  communication systems: {A} novel modeling technique,'' \emph{IEEE
  Communications Magazine}, vol.~60, no.~4, pp. 25--31, 2022.

\bibitem{tian2023satellite}
Y.~Tian, G.~Pan, H.~ElSawy, and M.-S. Alouini, ``Satellite-aerial
  communications with multi-aircraft interference,'' \emph{IEEE Transactions on
  Wireless Communications}, 2023, early access.

\bibitem{9079921}
N.~Okati, T.~Riihonen, D.~Korpi, I.~Angervuori, and R.~Wichman, ``Downlink
  coverage and rate analysis of low earth orbit satellite constellations using
  stochastic geometry,'' \emph{IEEE Transactions on Communications}, vol.~68,
  no.~8, pp. 5120--5134, 2020.

\bibitem{talgat2020stochastic}
A.~Talgat, M.~A. Kishk, and M.-S. Alouini, ``Stochastic geometry-based analysis
  of leo satellite communication systems,'' \emph{IEEE Communications Letters},
  vol.~25, no.~8, pp. 2458--2462, 2020.

\bibitem{wang2022conditional}
R.~Wang, A.~Talgat, M.~A. Kishk, and M.-S. Alouini, ``Conditional contact angle
  distribution in leo satellite-relayed transmission,'' \emph{IEEE
  Communications Letters}, vol.~26, no.~11, pp. 2735--2739, 2022.

\bibitem{wang2023reliability}
R.~Wang, M.~A. Kishk, and M.-S. Alouini, ``Reliability analysis of multi-hop
  routing in multi-tier leo satellite networks,'' \emph{IEEE Transactions on
  Wireless Communications}, 2023, early access.

\bibitem{belbase2018coverage}
K.~Belbase, Z.~Zhang, H.~Jiang, and C.~Tellambura, ``Coverage analysis of
  millimeter wave decode-and-forward networks with best relay selection,''
  \emph{IEEE Access}, vol.~6, pp. 22\,670--22\,683, 2018.

\bibitem{wang2022evaluating}
R.~Wang, M.~A. Kishk, and M.-S. Alouini, ``Evaluating the accuracy of
  stochastic geometry based models for {LEO} satellite networks analysis,''
  \emph{IEEE Communications Letters}, vol.~26, no.~10, pp. 2440--2444, 2022.

\bibitem{lou2023coverage}
Z.~Lou, B.~E.~Y. Belmekki, and M.-S. Alouini, ``Coverage analysis of hybrid
  {RF/THz} networks with best relay selection,'' \emph{IEEE Communications
  Letters}, 2023.

\bibitem{1198102}
A.~Abdi, W.~Lau, M.-S. Alouini, and M.~Kaveh, ``A new simple model for land
  mobile satellite channels: first- and second-order statistics,'' \emph{IEEE
  Transactions on Wireless Communications}, vol.~2, no.~3, pp. 519--528, 2003.

\bibitem{wang2022stochastic}
R.~Wang, M.~A. Kishk, and M.-S. Alouini, ``Stochastic geometry-based low
  latency routing in massive {LEO} satellite networks,'' \emph{IEEE
  Transactions on Aerospace and Electronic Systems}, vol.~58, no.~5, pp.
  3881--3894, 2022.

\bibitem{del2019technical}
I.~Del~Portillo, B.~G. Cameron, and E.~F. Crawley, ``A technical comparison of
  three low earth orbit satellite constellation systems to provide global
  broadband,'' \emph{Acta astronautica}, vol. 159, pp. 123--135, 2019.

\end{thebibliography}

\end{document}